\documentclass[12pt]{amsart}
\usepackage{amsfonts, amssymb, amsmath, amsthm, latexsym, array}
\usepackage{fullpage,color}
\usepackage{verbatim,euscript}

\numberwithin{equation}{section}

\usepackage{color}
\usepackage[colorlinks=true,linkcolor=blue,urlcolor=violet,citecolor=magenta]{hyperref}

\input {cyracc.def}

\newtheorem{thm}{Theorem}[section]
\newtheorem{qtn}{Problem}
\newtheorem{conj}[thm]{Conjecture}
\newtheorem{lm}[thm]{Lemma}
\newtheorem{cl}[thm]{Corollary}
\newtheorem{prop}[thm]{Proposition}

\theoremstyle{remark}
\newtheorem{rmk}[thm]{Remark}
\newtheorem*{rem}{Remark}

\theoremstyle{definition}
\newtheorem{ex}[thm]{Example}
\newtheorem*{exa}{Example}

\newcommand {\me}{{\mathfrak m}}

\newcommand {\gc}{{\mathcal C}}
\newcommand {\cf}{{\mathcal F}}

\newcommand {\N}{{\mathcal N}}

\newcommand {\cs}{{\mathcal S}}

\newcommand {\BZ}{{\mathbb Z}}
\newcommand {\BN}{{\mathbb N}}

\newcommand {\BC}{{\mathbb C}}

\newcommand{\odin}{{\mathrm{1\hspace{0.5pt}\!\! I}}}
\newcommand{\beq}{\begin{equation}}
\newcommand{\eeq}{\end{equation}}

\newcommand{\lb}{\lambda}
\newcommand{\ap}{\alpha}
\newcommand{\vp}{\varphi}

\renewcommand{\le}{\leqslant}
\renewcommand{\ge}{\geqslant}

\newcommand{\eus}{\EuScript}

\newcommand {\per}{{\mathrm{per}}}

\newcommand {\tr}{{\mathrm{tr\,}}}

\newcommand {\ov}{\overline}
\newcommand {\un}{\underline}

\newfam\eusfam%

\begin{document}
\setlength{\parskip}{3pt plus 5pt minus 0pt}
\hfill { {\color{cyan}\scriptsize \today}}
\vskip1ex

\title[Fredman's reciprocity and invariants of abelian groups]
{Fredman's reciprocity, invariants of abelian groups, and the permanent of the Cayley table}
\author[D.~Panyushev]{Dmitri I.~Panyushev}
\address[]{Independent University of Moscow,
Bol'shoi Vlasevskii per. 11, 119002 Moscow, \ Russia}
\email{panyush@mccme.ru}
\begin{abstract}
Let  $\eus R$ be the regular representation of a finite abelian group $G$ and let $\gc_n$
denote the cyclic group of order $n$.
For $G=\gc_n$,
we compute the Poincar\'e series of all $\gc_n$-isotypic components in
$\cs^{\cdot}\eus R\otimes \wedge^{\cdot}\eus R$
(the symmetric tensor exterior algebra of $\eus R$). 
From this we derive a general
reciprocity and some number-theoretic identities.
This generalises results of Fredman and Elashvili--Jibladze.
Then we consider the Cayley table, $\eus M_G$, of $G$ and some generalisations of it.
In particular, we prove that the number of formally different terms in the
permanent of $\eus M_G$ equals $(\cs^n \eus R)^G$, where $n$ is the order of $G$.
\end{abstract}
\maketitle

\section{Introduction}
\noindent
In the beginning of the 1970s, M.~Fredman \cite{fred} considered the 
problem of computing  the number of vectors 
$(\lb_0,\lb_1,\dots,\lb_{n-1})$ with non-negative integer components that satisfy
\beq  \label{eq:fred}
\lb_0+\dots +\lb_{n-1}=m \quad \text{ and } \quad \sum_{j=0}^{n-1} j\lb_j \equiv i \mod n .
\eeq
He denoted this number by $S(n,m,i)$.
Using generating functions, Fredman obtained an explicit formula for 
$S(n,m,i)$, which immediately showed that  
$S(n,m,i)=S(m,n,i)$. The latter is said to be {\it Fredman's reciprocity}.
Using a necklace interpretation, he also constructed a bijection  
between the vectors enumerated by $S(n,m,i)$
and those enumerated by $S(m,n,i)$. 
However, these results did not attract  attention and remained unnoticed.

Later, Elashvili and Jibladze \cite{elji-1,elji-2} (partly with Pataraia \cite{EJP})
rediscovered these  results using Invariant Theory.
Let $\gc_n\simeq \BZ/n\BZ$  be the cyclic group of order $n$ and $\eus R$ the space of its
regular representation over $\BC$.
Choose a basis $(v_0,v_1,\dots,v_{n-1})$ for $\eus R$ consisting of $\gc_n$-eigenvectors.
More precisely, if $\gamma\in\gc_n$ is a generator and  $\zeta=\sqrt[n]1$ a 
fixed primitive root of unity, then $\gamma{\cdot}v_i=\zeta^i v_i$. Write $\chi_i$ for the
linear character $\gc_n\to \BC^\times$ that takes $\gamma$ to $\zeta^i$. The
monomial $v_0^{\lb_1}\dots v_{n-1}^{\lb_{n-1}}$ has degree $m$ and weight $\chi_i$
if and only if $(\lb_0,\lb_1,\dots,\lb_{n-1})$ satisfies \eqref{eq:fred}.
Thus, $S(n,m,i)$ is the dimension of the space of $\gc_n$-semi-invariants of weight $\chi_i$
in the $m$th symmetric power $\cs^m \eus R$. 
This space can also be understood as
the $\gc_n$-isotypic component of type $\chi_i$ in $\cs^m \eus R$,
denoted  by $(\cs^m\eus R)_{\gc_n,\chi_i}$.
To stress the connection with cyclic groups, we will  write  
$a_i(\gc_n,m)$ in place of $S(n,m,i)$.
The celebrated Molien formula provides a closed expression for 
the generating function (Poincar\'e series)
\[
     \cf((\cs^{\cdot}\eus R)_{\gc_n,\chi_i};t)=\sum_{m=0}^\infty a_i(\gc_n,m)t^m , 
\]
where $(\cs^{\cdot}\eus R)_{\gc_n,\chi_i}=\bigoplus_{m\ge 0} (\cs^m\eus R)_{\gc_n,\chi_i}$
is the  $(\gc_n,\chi_i)$-isotypic component in $\cs^{\cdot}\eus R$.
Then extracting the coefficient of $t^m$ yields a 
formula for $a_i(\gc_n,m)$,  
see \eqref{f-ela-jib}.
It is worth stressing that 
Molien's formula is a very efficient tool that provides a uniform approach to various  
combinatorial problems and paves the way for further generalisations, see e.g. \cite{st79}.  

In this note, we elaborate on two topics.
{\sl First}, generalising results of Fredman and Elashvili-Jibladze, 
we compute the Poincar\'e series for each $\gc_n$-isotypic component in 
the bi-graded  
algebra $\cs^{\cdot}\eus R \otimes \wedge^{\cdot}\eus R$ and then
$\dim(\cs^{p}\eus R \otimes \wedge^{m}\eus R)_{\gc_n, \chi_i}$ for all $p,m,i$
(Theorem~\ref{thm:iso-comp}). From this we derive
a more general reciprocity, see \eqref{eq:gen-Hermite}.
As a by-product of these computations, we obtain some interesting identities, e.g.,
\[
     \exp(\frac{z}{1-z^2})=\prod_{d=1}^\infty (1+z^d)^{\vp(d)/d} ,
\] 
where $\vp$ is Euler's totient function.
In Section~\ref{sect:exterior},  several identities related to isotypic components in 
$\wedge^{\cdot}\eus R$ are given;
some of them are valid for an arbitrary finite abelian group $G$, see 
Theorem~\ref{thm:sum=0}.
{\sl Second}, in Section~\ref{sect:Keli-table}, 
we study some properties of the Cayley table, $\eus M_G$, of $G$. If 
$G=\{x_0,x_1,\dots,x_{n-1}\}$, then  $\eus M_G$ can be regarded
as $n$ by $n$  matrix with entries in $\BC[x_0,\dots,x_{n-1}]\simeq \cs^{\cdot}\eus R$. 
For $G=\gc_n$, \ $\eus M_G$ is nothing but a generic {\it circulant matrix}. 
The permanent of $\eus M_G$,  $\per(\eus M_G)$, is a sum of monomials in $x_i$'s 
of degree $n$. Using 
\cite{hall}, we prove that the number of different monomials 
occurring in this sum equals $\dim(\cs^n\eus R)^G$. 
Then we introduce the extended Cayley table, $\widetilde{\eus M}_G$ (which is a matrix of order $n+1$), and characterise the monomials occurring in $\per(\widetilde{\eus M}_G)$ 
(Theorem~\ref{thm:per-ext-keli}).
This characterisation implies that the number of different monomials in
$\per(\widetilde{\eus M}_G)$ equals $\dim(\cs^{n+1}\eus R)^G$. 
Both $\per(\eus M_G)$ and $\det(\eus M_G)$ belong to $\cs^n\eus R$, and we prove that $\per(\eus M_G)$ is $G$-invariant, whereas $\det(\eus M_G)$ is a semi-invariant whose 
weight is the sum of all elements of the dual group $\hat G$. The latter means that in many cases $\det(\eus M_G)$ is invariant, too. In Section~\ref{sect:questions}, 
we  discuss some open problems related to  $(\cs^{\cdot}\eus R)^G$ and
$\per(\eus M_G)$.

\noindent
\un{Notation:}  $\# (M)$ is the cardinality of a finite set $M$; 
\ $(n,m)$ is the greatest common divisor of $n,m\in \BN$; $G$ is always a finite group.

 {\small  {\bf Acknowledgements.}
A part of this work was done during my 
stay at the Centro di Ricerca Matematica Ennio De Giorgi (Pisa) in July 2009. 
I am grateful to this institution for the warm hospitality and support.}

\section{Preliminaries} 
\label{prelim}

\subsection{Ramanujan's sums}
Two  important number-theoretic functions are
{\it Euler's totient function\/} $\varphi$  and the {\it M\"obius function} $\mu$.
Recall that $\vp(n)$ is the number of all primitive roots of unity of order $n$. 
Given $i, n\in\BN$, $n\ge 1$, the {\it Ramanujan's sum}, 
$c_n(i)$,  is the sum of $i$-th powers of the primitive roots of unity of order $n$.
In particular, $c_n(0)=\vp(n)$.
There are two useful expressions for  Ramanujan's sums:
\[ 
c_n(i)= \sum_{d \vert (n,i)} \mu\left(\displaystyle \frac{n}{d}\right)d  ,  \qquad
c_n(i)= \frac{\vp(n)}{\vp\left(\displaystyle\frac{n}{(n,i)}\right)} 
{\cdot} \mu\left({\frac{n}{(n,i)}}\right),
\] 
see \cite[Theorems 271 \& 272]{ha-wr}.
These formulae also show that  $c_n(1)=\mu(n)$, $c_n(i)=c_n(n{-}i)$, and
$c_n(i)$ is always a rational  integer.  

\subsection{Molien's formula for the symmetric algebra} 
\label{subs:Molien} 
Let $G$ be a finite group and  $V$ a finite-dimensional $G$-module. The original
Molien formula computes the Poincar\'e series of the graded algebra of invariants  
$(\cs^{\cdot}V)^G=\bigoplus_{m\ge 0} (\cs^m V)^G$.
More generally, there is a similar 
formula for the Poincar\'e series of any $G$-isotypic component  in $\cs^{\cdot} V$. 
Let $\chi$ be an irreducible representation of $G$ and 
$(\cs^{\cdot} V)_{G,\chi}$  the isotypic component of type $\chi$  in 
$\cs^{\cdot} V$. By definition, the Poincar\'e series of  $(\cs^{\cdot} V)_{G,\chi}$ is the power series
$\cf((\cs^{\cdot} V)_{G,\chi};t):=\sum_{m\ge 0}\dim \bigl( (\cs^m V)_{G,\chi}\bigr) t^m$.
Then 
\[
\cf\bigl((\cs^{\cdot} V)_{G,\chi};t\bigr)=\frac{\deg(\chi)}{\#(G)}\sum_{\gamma\in G}
 \frac{\tr(\chi(\gamma^{-1}))}{\det_V(\odin-t\gamma )} ,
\]
see e.g. \cite[Thm.\,2.1]{st79}. Here $\odin$ is the identity matrix in $GL(V)$. (The algebra of invariants corresponds to the trivial one-dimen\-si\-onal representation, i.e., if $\deg(\chi)=1$ and $\chi(\gamma)=1$ for all $\gamma\in G$.)

Let $\eus R$ be the space of the regular representation of $G$. 
For the $G$-module $\eus R$, Molien's formula 
can be presented in a somewhat simpler form. 

\begin{prop}[\protect{\cite[V.1.8]{AF76}}]  \label{molien_reg}
Let $\vp_G(d)$ be the number of elements of order $d$ in $G$. Then
\[
\displaystyle \cf((\cs^{\cdot}\eus R)^{G};t)=\frac{1}{\#(G)}\sum_{d\ge 1} 
\frac{\vp_G(d)}{(1-t^d)^{(\# G)/d}} . 
\]
\end{prop}
\noindent
This can easily be extended to an arbitrary $\chi$. If $\mathsf{ord}(\gamma)$ is the order
of $\gamma\in G$, then 
\beq    \label{eq:molin-reg-chi}
 \cf((\cs^{\cdot}\eus R)_{G,\chi};t)=
\frac{\deg(\chi)}{\#(G)}\sum_{d\vert \# G} 
\frac{\sum_{\gamma:\,\mathsf{ord}(\gamma)=d}\tr(\chi(\gamma^{-1}))}{(1-t^d)^{(\# G)/d}} .
\eeq
In fact, we prove below  a more general formula (Lemma~\ref{lem:gen-formula}).

\subsection{Formulae of Fredman and Elashvili-Jibladze} 
\label{subs:el-ji} 
\noindent
Recall that $a_i(\gc_n,m)=\dim \cs^m(\eus R)_{\gc_n,\chi_i}$ or, equivalently,
it is the number of vectors satisfying \eqref{eq:fred}.
In particular,  $a_0(\gc_n,m)=\dim \cs^m(\eus R)^{\gc_n}$.
If the elements  of $\gc_n$ are regarded as the roots of unity of order $n$,
then $\chi_i$ is the character $\xi\mapsto \xi^i,\ \xi\in\gc_n$.
Here $\vp_{\gc_n}(d)$  is almost Euler's totient function.
That is, $\vp_{\gc_n}(d)=\vp(d)$, if $d\vert n$; and $\vp_{\gc_n}(d)=0$ otherwise.
Using \eqref{eq:molin-reg-chi} with $G=\gc_n$ and $\chi=\chi_i$, 
we see that $\deg(\chi_i)=1$ and 
$\sum_{\gamma:\,\mathsf{ord}(\gamma)=d}\chi_i(\gamma^{-1})=c_d(d-i)$.
Then extracting the coefficient
of $t^m$ yields a nice-looking formula
(Fredman \protect\cite{fred}, Elashvili-Jibladze \protect\cite{elji-2})
\begin{equation}  \label{f-ela-jib}
   a_i(\gc_n,m)=\frac{1}{n+m}\sum_{d\vert (n,m)}c_d(i)
   \genfrac{(}{)}{0pt}{}{n/d+m/d}{n/d} \ .
\end{equation}

\begin{rmk}
Both Fredman's approach, see \eqref{eq:fred}, and cyclic group interpretation presuppose
that $a_i(\gc_n,m)$ is defined for $n\ge 1$ and $m\ge 0$. But \eqref{f-ela-jib} shows that 
$a_i(\gc_n,m)$ is naturally defined for $(n,m)\in\BN^2$, $(n,m)\ne (0,0)$. 
\end{rmk}
\noindent
It follows from \eqref{f-ela-jib} that $a_i(\gc_n,m)=a_i(\gc_m,n)$.
In  \cite{elji-1,elji-2,EJP},
this equality is named the ``Hermite reciprocity''. As it  has no relation to Hermite and
was first proved by Fredman,
the term {\it Fredman's reciprocity\/} seems to be more appropriate.

From  \eqref{f-ela-jib}, one  can derive the equality 
\beq \label{rmk:log-elji}
  \sum_{(n,m)\in \BN^2,\,(n,m)\ne (0,0)} 
  a_i(\gc_n,m)x^ny^m=-\sum_{d=1}^\infty \frac{c_d(i)}{d}\log(1-x^d-y^d) \ .
\eeq
(Cf. \cite[Remark\,2]{elji-1}, \cite[Sect.\,4]{EJP}.)

\section{Symmetric tensor exterior algebra and Poincar\'e series}
\label{sect:symm-exterior}

\noindent
As above, let $V$ be a $G$-module.
We consider the Poincar\'e series of the $G$-isotypic components in 
$\cs^{\cdot} V\otimes \wedge^{\cdot}V$.
Let $(\cs^{\cdot} V\otimes \wedge^{\cdot}V )_{G,\chi}$ denote the isotypic component
corresponding to an irreducible representation $\chi$. It is a bi-graded vector  space and its 
Poincar\'e series  is the formal  power series
\[
   \cf\bigl((\cs^{\cdot} V\otimes \wedge^{\cdot}V )_{G,\chi}; s,t\bigr)=
   \sum_{p,m\ge 0} \dim(\cs^p V\otimes \wedge^m V)_{G,\chi}\,s^pt^m .
\]
(Clearly, it is a polynomial with respect to  $t$.) It is known that  
\[
 \cf\bigl((\cs^{\cdot} V\otimes \wedge^{\cdot}V )^{G}; s,t\bigr)=
 \frac{1}{\# G}\sum_{\gamma\in G} \frac{\det_V(\odin+t\gamma )}{\det_V(\odin-s\gamma )},
\]
see \cite[Theorem\,1.33]{alm82}. A similar argument provides the formula for an arbitrary 
$G$-iso\-ty\-pic component:
\beq    \label{eq:gen-iso-comp}
\cf\bigl((\cs^{\cdot} V\otimes \wedge^{\cdot}V )_{G,\chi}; s,t\bigr)=
 \frac{\deg(\chi)}{\# G}\sum_{\gamma\in G}\tr(\chi(\gamma^{-1})) \frac{\det_V(\odin+t\gamma )}{\det_V(\odin-s\gamma )}.
\eeq
For, in place of the Reynolds operator  
$\displaystyle\frac{1}{\#(G)}\sum_{\gamma\in G}\gamma$ \ (the projection to the subspace of $G$-invariants), one
should merely exploit the operator  $\displaystyle\frac{\deg(\chi)}{\#(G)}\sum_{\gamma\in G}\tr(\chi(\gamma^{-1}))\gamma$ \ (the projection to the isotypic component of type $\chi$).

\begin{lm}    \label{lem:gen-formula}
For the regular representation $\eus R$ of $G$, 
the right-hand side of \eqref{eq:gen-iso-comp} can be written as 
\[ 
 \frac{\deg(\chi)}{\# G}\sum_{d\ge 1} \left(
 \sum_{\gamma:\,\mathsf{ord}(\gamma)=d}\tr(\chi(\gamma^{-1})){\cdot}
 \left(\frac{1-(-t)^d)}{1-s^d}\right)^{(\# G)/d}
\right) .
\]
\end{lm}\begin{proof}
If $\gamma\in G$ is of order $d$, then  $\langle \gamma\rangle\simeq \gc_d$ and 
 each coset of $\langle \gamma\rangle$  in $G$ is a cycle of 
length $d$ with respect to the multiplication by  $\gamma$. 
Hence, in a suitable basis of $\eus R$, the matrix of $\gamma$
in $GL(\eus R)$
consists of $(\# G)/d$ diagonal \\blocks 
$   \begin{pmatrix} 
      0  & 1 & 0 & \dots   & 0 \\
      0  & 0 & 1 & \ddots & 0 \\
      0  & 0 &  \ddots & \ddots & 0 \\
      0  & 0 & \ddots  & \ddots & 1 \\   
      1  & 0 & 0 & \dots & 0    
 \end{pmatrix}$ 
of size $d$. Since 
$\det\begin{bmatrix} 
      1 & -s & 0 & \dots & 0 \\
      0  & 1 & -s & \ddots & 0 \\
       0 &  0&  \ddots & \ddots & 0 \\
      0  & 0 & \ddots  & \ddots & -s \\   
      -s & 0 & 0 & \dots & 1    
\end{bmatrix}=1-s^d$, 
we obtain  $\displaystyle \frac{\det_{\eus R}(\odin+t\gamma )}{\det_{\eus R}(\odin-s\gamma )}=
\left(\frac{1-(-t)^d)}{1-s^d}\right)^{(\# G)/d}$, which proves the lemma.
\end{proof}

Now, we apply this lemma to the regular representation of $\gc_n$.
Recall that the number  $\displaystyle \genfrac{(}{)}{0pt}{}{a+b+c}{a,\, b,\, c}$ is 
defined to be $\displaystyle \frac{(a+b+c)!}{a!\, b!\, c!}$.

\begin{thm}  \label{thm:iso-comp}
The Poincar\'e series of the $(\gc_n,\chi_i)$-isotypic component equals
\begin{multline}  \label{eq:cyclic-poinc}
\cf\bigl((\cs^{\cdot}\eus R \otimes \wedge^{\cdot}\eus R )_{\gc_n,\chi_i}; s,t\bigr)=
\frac{1}{n}\sum_{d\vert n} c_d(i)\frac{(1-(-t)^d)^{n/d}}{(1-s^d)^{n/d}}
\\  
=\frac{1}{n}\sum_{d\vert n} c_d(i)
\Bigl(\sum_{a=0}^{n/d} (-1)^{(d+1)a}\genfrac{(}{)}{0pt}{}{n/d}{a}t^{ad} \Bigr)
\Bigl( \sum_{b\ge 0} \genfrac{(}{)}{0pt}{}{(n/d)+b-1}{(n/d)-1}s^{bd}\Bigr) .
\end{multline}
Consequently, 
\begin{equation}  \label{dim-nonsym}
   \dim \bigl(\cs^p \eus R\otimes \wedge^m \eus R \bigr)_{\gc_n,\chi_i}=\frac{(-1)^m}{p+n}
  \sum_{d\vert n,p,m} (-1)^{m/d} c_d(i) \genfrac{(}{)}{0pt}{}{(n+p)/d}{m/d,\, p/d,\, (n{-}m)/d}  .
\end{equation}
\end{thm}
\begin{proof}
This  is a straightforward consequence of Lemma~\ref{lem:gen-formula}.
If $G=\gc_n$, then $\deg(\chi_i)=1$ and $\sum_{\gamma:\,\mathsf{ord}(\gamma)=d}\chi_i(\gamma^{-1})=
c_d(n{-}i)=c_d(i)$, which proves \eqref{eq:cyclic-poinc}.

\noindent
We leave it to the reader to extract the coefficient of $t^m s^p$ in \eqref{eq:cyclic-poinc} 
and obtain \eqref{dim-nonsym}.
\end{proof}

Letting $n=q+m$ yields a more symmetric form of \eqref{dim-nonsym}:
\begin{equation}  \label{eq:dim-sym}
   \dim \bigl(\cs^p \eus R\otimes \wedge^m \eus R\bigr)_{\gc_{q+m},\chi_i}=\frac{(-1)^m}{p+q+m}
  \sum_{d\vert p,q,m} (-1)^{m/d} c_d(i) \genfrac{(}{)}{0pt}{}{(m+p+q)/d}{m/d,\, p/d,\, q/d}  .
\end{equation}
As the right-hand side is  symmetric with respect to $p$ and $q$, we get an equality 
for dimensions of isotypic components related to the regular representations
of two cyclic groups, $(\gc_{q+m},\eus R)$ and $(\gc_{p+m},\tilde{\eus R})$:
\begin{equation}  \label{eq:gen-Hermite}
\dim \bigl(\cs^p \eus R\otimes \wedge^m \eus R\bigr)_{\gc_{q+m},\chi_i}=
\dim \bigl(\cs^q \tilde{\eus R}\otimes \wedge^m \tilde{\eus R}\bigr)_{\gc_{p+m},\chi_i}.
\end{equation}
For $m=0$, this simplifies to Fredman's reciprocity \cite[(4)]{fred}.
It would be interesting to have a combinatorial interpretation of this symmetry in the
spirit of Fredman's approach.
 
\begin{rmk} 
(1) \ Letting $t=0$ in \eqref{eq:cyclic-poinc} or $m=0$ in \eqref{dim-nonsym}, we get known formulae for the isotypic components in the symmetric algebra of $\eus R$, see \cite{elji-1, elji-2}.
Letting $s=0$ in \eqref{eq:cyclic-poinc} or $n=0$ in \eqref{dim-nonsym}, we get 
interesting formulae for the isotypic components in the exterior algebra of $\eus R$, see 
the next section.

(2) \ If $d$ is always odd (e.g. at least one of $m,  p, q$ is odd), then
$(-1)^{m+\frac{m}{d}}=1$ and the right-hand side
of \eqref{eq:dim-sym} becomes totally symmetric with respect to $p,q,m$.
\end{rmk}

The following is  a generalisation of  \eqref{rmk:log-elji}:

\begin{prop}   \label{prop:log}
\[
  \sum_{(p,q,m)\in \BN^3,\,p{+}q{+}m\ge 1} \dim
  \bigl(\cs^p \eus R\otimes \wedge^m \eus R\bigr)_{\gc_{q+m},\chi_i}{\cdot} x^py^q z^m=-\sum_{d=1}^\infty \frac{c_d(i)}{d}\log(1-x^d-y^d+(-z)^d) \ .
\]
\end{prop}\begin{proof}
By \eqref{eq:dim-sym}, the left-hand side equals
\[
\sum_{p+q+m\ge 1}\frac{(-1)^m}{p+q+m}\sum_{d\vert p,q,m} (-1)^{m/d}c_d(i) \genfrac{(}{)}{0pt}{}{(p{+}q{+}m)/d}{p/d,\, q/d,\, m/d}x^py^qz^m .
\]
Letting $p/d=\ap,\, q/d=\beta,\, m/d=\gamma$, we rewrite it as
\begin{multline*}
  \sum_{d=1}^\infty \frac{c_d(i)}{d}\sum_{\ap+\beta+\gamma\ge 1}
  \frac{(-1)^\gamma}{\ap+\beta+\gamma}\genfrac{(}{)}{0pt}{}{\ap+\beta+\gamma}{\ap,\, \beta,\, \gamma}x^{\ap d}y^{\beta d}(-z)^{\gamma d}=
  \\
  \sum_{d=1}^\infty \frac{c_d(i)}{d}\sum_{k\ge 1}\Bigl(\sum_{\ap+\beta+\gamma=k}
  \frac{1}{k}\genfrac{(}{)}{0pt}{}{k}{\ap,\, \beta,\, \gamma}(x^d)^\ap(y^d)^\beta (-(-z)^d)^\gamma\Bigr)=\\
  \sum_{d=1}^\infty \frac{c_d(i)}{d}\sum_{k\ge 1} \frac{(x^d+y^d-(-z)^d)^k}{k}=
  -\sum_{d=1}^\infty \frac{c_d(i)}{d}\log(1-x^d-y^d+(-z)^d).
\end{multline*}
\end{proof}
Specializing the equality of Proposition~\ref{prop:log}, we get some interesting identities.
\\
A) Taking $x=y=0$ forces that $p=q=0$ in the left-hand side, which leads to the equality
\[
   \sum_{m\ge 1}\dim (\wedge^m\eus R)_{\gc_m,\chi_i}z^m =-
   \sum_{d=1}^\infty \frac{c_d(i)}{d}\log(1+(-z)^d).
\]
For $i=0$,  we have $c_d(0)=\vp(d)$ and $\dim (\wedge^m\eus R)^{\gc_m}=\begin{cases}  
1, & m \text{ odd }; \\ 0, & m \text{ even}.
\end{cases}$ \\  
    \text{That is, } $\displaystyle\frac{z}{1-z^2}=-\sum_{d=1}^\infty \frac{\vp(d)}{d}\log(1+(-z)^d).$
Replacing $z$ with $-z$ and exponentiating, we finally obtain:
\[
     \exp\left(\frac{z}{1-z^2}\right)=\prod_{d\ge 1} (1+z^d)^{\vp(d)/d} .
\]

\noindent
B) \  Likewise, for $x=z=0$ (or just $x=0$ in \eqref{rmk:log-elji}), we get
\[
    -\sum_{d=1}^\infty \frac{c_d(i)}{d}\log(1-y^d)=\begin{cases} y/(1-y), & i=0 \\
    0, & i\ne 0 .
    \end{cases}
\]
In particular, 
\[
   \exp\left(\frac{-y}{1-y}\right)=\prod_{d\ge 1} (1-y^d)^{\vp(d)/d} .
\]
\section{On the exterior algebra of the regular representation}
\label{sect:exterior}

\noindent
In case of  the exterior algebra of a $G$-module, 
the Poincar\'e series of an isotypic component is actually a polynomial
in $t$, which can be evaluated for any  $t$. Here we gather some practical formulae for the regular representations and for cyclic  groups.

First, using \eqref{eq:gen-iso-comp} and Lemma~\ref{lem:gen-formula} with trivial $\chi$ and
$s=0$, we obtain
\[
    \cf((\wedge^{\cdot}\eus R)^G;t)=\frac{1}{\#(G)}\sum_{d\ge 1} \vp_G(d)(1-(-t)^d)^{\#(G)/d} .
\]
It follows that  $\cf((\wedge^{\cdot}\eus R)^G;t)$ always has the factor $1+t$ and
\beq    \label{eq:dim-ext-inv}
  \dim(\wedge^{\cdot}\eus R)^G=\frac{1}{\#(G)}\sum_{ d\text{ odd}} \vp_G(d)2^{\#(G)/d} . 
\eeq
Note that here $G$ is not necessarily abelian!
\begin{ex}
For $G=\eus S_3$, we have $\vp_G(1)=1$, $\vp_G(2)=3$, and $\vp_G(3)=2$.
Therefore,
\[
  \cf((\wedge^{\cdot}\eus R)^{\eus S_3};t)=
  \frac{1}{6}\bigl( (1+t)^6+3(1-t^2)^3+2(1+t^3)^2\bigr)=1+t+t^2+4t^3+4t^4+t^5 .
\]
\end{ex}
For $G=\gc_n$, there are  precise assertions for all $G$-isotypic components in 
$\wedge^{\cdot}\eus R$.
Using Theorem~\ref{thm:iso-comp} with $s=0$ and $p=0$, we obtain
\begin{gather}  \label{eq:poinc-ext}
    \cf((\wedge^{\cdot}\eus R)_{\gc_n,\chi_i}; t )=
\frac{1}{n}\sum_{d\vert n} c_d(i)(1-(-t)^d)^{n/d}, \\   \label{eq:dim-ext}
\dim \bigl((\wedge^m\eus R)_{\gc_n,\chi_i}\bigr)=\frac{(-1)^m}{n}
  \sum_{d\vert n,m} (-1)^{m/d} c_d(i) \genfrac{(}{)}{0pt}{}{n/d}{m/d}=: b_i(\gc_n,m) . 
\end{gather}
Again, it is convenient to replace $n$ with $q+m$ in \eqref{eq:dim-ext}. Then
\[
b_i(\gc_{q+m},m)=\dim \bigl((\wedge^m\eus R)_{\gc_{q+m},\chi_i}\bigr)=\frac{(-1)^m}{q+m}
  \sum_{d\vert q,m} (-1)^{m/d} c_d(i) \genfrac{(}{)}{0pt}{}{q/d+m/d}{m/d} . 
\]
From this we derive the following observation:
\begin{prop}
If $q$ or $m$ is odd, then $b_i(\gc_{q+m}, m)=a_i(\gc_{q}, m)$ and also
$b_i(\gc_{q+m}, m)=b_i(\gc_{q+m}, q)$.
\end{prop}

\begin{ex}
$b_i(\gc_{2n-1},n-1)=a_i(\gc_{n-1},n)$, and it is the $(n-1)$-th Catalan number regardless
of $i$.
\end{ex}

\begin{rem}
If $n$ is odd, then $\wedge^n \eus R$ is the trivial $\gc_n$-module and therefore
$\wedge^m \eus R \simeq \wedge^{n-m} \eus R$ as $\gc_n$-modules.  
This ``explains'' the equality $b_i(\gc_{n}, m)=b_i(\gc_{n}, n-m)$ for $n$ odd.
\end{rem}

Substituting $t=1$ in \eqref{eq:poinc-ext} yields a nice formula for dimension of the whole isotypic component:
\begin{equation}  \label{dim-ext-comp}
   \dim \bigl((\wedge^{\cdot}\eus R)_{\gc_n,\chi_i}\bigr)=
   \frac{1}{n} \sum_{d\vert n,\ d\ \text{odd}}  c_d(i)\, 2^{n/d} .
\end{equation}
For $i=0$, this becomes a special case of \eqref{eq:dim-ext-inv}. 
There is a down-to-earth interpretation of \eqref{dim-ext-comp} that does not invoke  
Invariant Theory.  As in the introduction,
choose a basis $\{v_0,v_1,\dots,v_{n-1}\}$ for $\eus R$ such that $v_i$ has weight $\chi_i$.
Then 
\[
v_{j_1}\wedge\dots \wedge v_{j_m}\in (\wedge^m\eus R)_{\gc_n,\chi_i} \quad 
\Longleftrightarrow \quad j_1+\dots + j_m \equiv i \mod n .
\]
Consequently, $\dim(\wedge^{\cdot}\eus R)_{\gc_n,\chi_i}$ equals the number of subsets
$J\subset \{0,1,\dots,n{-}1\}$ such that $|J| \equiv i \mod n$.
(Here $|J|$ stands for the sum of elements of $J$.) Hence our invariant-theoretic 
approach proves the following purely combinatorial fact:
\[
   \#\{ J\subset \{0,1,\dots,n{-}1\} \mid |J| \equiv i \mod n \}  = \frac{1}{n} \sum_{d\vert n,\ d\ \text{odd}}  c_d(i)\, 2^{n/d} .
\]
For $i=0$, this is nothing but the number of subsets of $\gc_n$ summing to the neutral element  (in the additive notation).    
We provide a similar interpretation for any  abelian group.  

\begin{thm}   \label{thm:sum=0}
For an abelian group  $G$, let $\N_G$ denote the number of subsets $S$ of $G$
such that $|S|:=\sum_{\gamma\in S}\gamma=0\in G$. Then
$\displaystyle \N_G=\dim(\wedge^{\cdot}\eus R)^G=\frac{1}{\#(G)}\sum_{d\text{ \rm odd}} \vp_G(d)2^{\#(G)/d}$. 
\end{thm}
\begin{proof}
In view of \eqref{eq:dim-ext-inv}, only the first equality requires a proof.
Let $(z_0,\dots,z_{n-1})$ be a basis for $\eus R$ consisting of $G$-eigenvectors,
$n=\#(G)$. Here the weight of $z_i$ is a linear character $\chi_i$ and
$\hat{G}=\{\chi_0,\chi_1,\dots,\chi_{n-1}\}$ is the dual group of $G$. 
One of the $\chi_i$'s is the neutral
element of $\hat G$, denoted by $\hat 0$ in the additive notation.
Then
\[
z_{j_1}\wedge\dots \wedge z_{j_m}\in (\wedge^m\eus R)^G \quad 
\Longleftrightarrow \quad \chi_{j_1}+\dots + \chi_{j_m}=\hat 0\in \hat G.    
\]
Thus, $\dim(\wedge^{\cdot}\eus R)^G$ equals the number of subsets of $\hat G$ summing to $\hat 0$.
However, the groups $\hat G$ and $G$ are (non-canonically) isomorphic, hence
$\N_G=\N_{\hat G}$ and we are done. 
\end{proof}

\section{On the permanent of the Cayley table of an abelian group}
\label{sect:Keli-table}

\noindent In this section, $G$ is an abelian  group, 
$G=\{x_0,x_1,\dots,x_{n-1}\}$. 
The {\it Cayley table\/} of $G$, denoted $\eus M_G=(m_{i,j})$, 
can be regarded as $n$ by $n$ matrix with entries in
the polynomial ring $\BC[x_0,x_1,\dots,x_{n-1}]\simeq \cs^{\cdot}\eus R$. 
To distinguish the addition in $\BC[x_0,x_1,\dots,x_{n-1}]$ and the group 
operation in $G$, the latter is denoted by 
`$\dotplus $'.
By definiton, $m_{i,j}=x_i\dotplus x_j$, $i,j=0,\dots, n-1$.  
Hence $\eus M_G$ is a symmetric matrix.
The permanent of $\eus M_G$, $\per(\eus M_G)$, is 
a homogeneous polynomial of degree $n$ in $x_i$'s, and
it does not depend on the ordering of  elements of $G$.
Let $p(G)$ denote the number of formally different monomials occurring in  $\per(\eus M_{G})$.

\begin{rmk}    
In place of the Cayley table, one can consider the matrix $\hat{\eus M}_G$ with entries
$\hat m_{i,j}=x_i\circleddash x_j$ (the difference in $G$). Clearly, $\hat{\eus M}_G$ 
is obtained from $\eus M_G$ by rearranging the columns only (or, the rows only), using the
permutation on $G$ that takes each element to its inverse. Therefore
$\per(\hat{\eus M}_G)=\per(\eus M_{G})$ and 
$\det(\hat{\eus M}_G)=\pm\det(\eus M_{G})$. Although $\hat{\eus M}_G$ is not 
symmetric in general, 
an advantage is that every entry on the main diagonal is the neutral elements of $G$.
\end{rmk}

\begin{ex}   \label{ex:circulant}
For $G=\gc_n$ and the natural ordering of its elements (i.e., $x_i$ corresponds to $i$), 
one obtains a generic {\it circulant
matrix} (the latter means that the rows 
are successive cyclic permutations of the first row). More precisely, 
$\eus M_{\gc_n}$  (resp. $\hat{\eus M}_{\gc_n}$)
is a circulant matrix in Hankel (resp. Toeplitz) form.
For instance,
$\eus M_{\gc_3}=\begin{pmatrix}  x_0 & x_1 & x_2 \\
x_1 & x_2 & x_0 \\x_2 & x_0 & x_1 \\
\end{pmatrix}$.  Here 
 $\per(\eus M_{\gc_3})=x_0^3+x_1^3+x_2^3+3x_0x_1x_2$. Therefore $p(\gc_3)=4$.
\end{ex}

The function $n\mapsto p(\gc_n)$ was  studied in \cite{bru-new}, where it was pointed out 
that the main result of 
Hall~\cite{hall} shows that  $p(\gc_n)$ equals the number of 
solutions to  
\[  
  \begin{cases}\lb_0+\dots +\lb_{n-1}=n,  
  \\ \displaystyle \sum_{j=0}^{n-1} j\lb_j \equiv 0 \mod n .\end{cases} 
\] 
That is, $p(\gc_n)=a_0(\gc_n,n)$ in our notation. 
Because results of \cite{hall} apply to arbitrary finite abelian groups, one can be interested in 
$p(G)$ in this more general setting.  Below, we give
an invariant-theoretic answer using that result of Hall.

Let $\eus S_n$ denote the symmetric group acting by permutations on
$\{0,1,\dots,n-1\}$. Accordingly, $\eus S_n$ permutes the elements of $G$ by the rule
$\pi(x_i):=x_{\pi(i)}$.
Recall that 
\[
\per(m_{i,j})=\sum_{\pi\in\eus S_n}\prod_{i=0}^{n-1}m_{i,\pi(i)} .
\]
For the matrix $\eus M_G$, 
\[
   \prod_{i=0}^{n-1}m_{i,\pi(i)}=\prod_{i=0}^{n-1}(x_i\dotplus x_{\pi(i)})=
   \prod_{i=0}^{n-1}x_i^{k_i(\pi)}=:\boldsymbol{x}(\pi)
\]
is a monomial in $x_i$'s of degree $n$. Note that different permutations may result in 
the same monomial. 
The following is essentially proved by M.~Hall.

\begin{thm}[\protect{\cite[n.\,3]{hall}}]    
\label{thm:hall-perm}
A monomial\/ $\me=\prod_{i=0}^{n-1}x_i^{k_i}$ is of the form $\boldsymbol{x}(\pi)$ for
some $\pi\in\eus S_n$ (i.e., occurs in $\per(\eus M_G)$) if and only if\/
$ \sum_i k_i=n$  \text{ and } 
$k_0x_0\dotplus \dots \dotplus k_{n-1}x_{n-1}= 0\in G$.
{\rm [Of course, here $k_i x_i$ stands for  $x_i\dotplus \dots\dotplus x_i$  ($k_i$ times). ]}
\end{thm}

\noindent
The necessity of the conditions is easy; a non-trivial argument is required for the sufficiency, 
i.e., for the existence of $\pi$.

\begin{thm}   \label{thm:number-in-perm}
$p(G)=\dim \cs^n(\eus R)^G$.
\end{thm}
\begin{proof}
Let $(z_0,\dots,z_{n-1})$ be a basis for $\eus R$ consisting of $G$-eigenvectors. Recall
that the weight of $z_i$ is $\chi_i$ and
$\hat{G}=\{\chi_0,\dots,\chi_{n-1}\}$ is the dual group. 
The monomial $z_0^{k_0}\dots z_{n-1}^{k_{n-1}}\in \cs^{\cdot}\eus R$ 
is a semi-invariant of $G$ of weight 
$k_0\chi_0\dotplus \dots \dotplus k_{n-1}\chi_{n-1}\in \hat G$.
It follows that 
\[
  \dim \cs^n(\eus R)^G=\{(k_0,\dots,k_{n-1}) \mid \sum_i k_i=n \ \ \&  \ \ 
  k_0\chi_0\dotplus \dots \dotplus k_{n-1}\chi_{n-1}=\hat 0\} .
\]
Modulo the passage from $G$ to $\hat G$, these conditions coincide with those of 
Theorem~\ref{thm:hall-perm}. Since $G\simeq \hat G$, we are done.
\end{proof}

Our next goal is to extend these results to a certain matrix of order $n+1$.  
We begin with two assertionts on $\per(\eus M_G)$, which are of independent interest.

\begin{prop}   \label{prop:nat-action}
There is a natural action $\ast: G\times \eus S_n\to \eus S_n$ such that, for  
$\gamma\in G$ and 
$\pi\in\eus S_n$,  $\text{\rm sign}(\gamma{\ast}\pi)=\text{\rm sign}(\pi)$ and
$\boldsymbol{x}(\gamma{\ast}\pi)=\boldsymbol{x}(\pi)$.
\end{prop}
\begin{proof}
Every $\gamma\in G$ determines a permutation  $\sigma_\gamma$ 
on $G$ and thereby an element of $\eus S_n$. Namely:
\[
   (x_0,\dots,x_{n-1})  \ \stackrel{\sigma_\gamma}{\mapsto} \ (\gamma\dotplus x_0,\dots,
   \gamma\dotplus x_{n-1}) .
\]
Equivalently,   $x_{\sigma_\gamma(i)}=x_i\dotplus \gamma$.
Define the $G$-action on $\eus S_n$ by   
$\gamma{\ast}\pi=\sigma_\gamma \pi \sigma_\gamma$.  Hence $\text{\rm sign}(\gamma\ast\pi)=\text{\rm sign}(\pi)$. 
Recall that $\boldsymbol{x}(\pi)=\prod_{i=0}^{n-1}(x_i\dotplus x_{\pi(i)})$.
Then
\[
  \boldsymbol{x}(\gamma{\ast}\pi)=\prod_{i=0}^{n-1}(x_i\dotplus 
  x_{\sigma_\gamma\pi\sigma_\gamma(i)})=\prod_{j=0}^{n-1}(x_{\sigma_\gamma^{-1}(j)}\dotplus x_{\sigma_\gamma\pi(j)}) ,
\]
where $j=\sigma_\gamma(i)$.  By definition, 
$x_{\sigma_\gamma\pi(j)}=x_{\pi(j)}\dotplus \gamma$ and $x_j=x_{\sigma_\gamma^{-1}(j)}
\dotplus \gamma$. Thus, the linear  factors of $\boldsymbol{x}(\gamma{\ast}\pi)$ 
remain the same.
\end{proof}

\begin{rmk}
Our action `$\ast$' can be regarded as a generalisation of Lehmer's ``operator $S$'' for circulant matrices~\cite[p.\,45]{lehm}, i.e., essentially, for $G=\gc_n$. Using that operator Lehmer proved  that, for $n=p$ odd prime,
\[
   \det(\eus M_{\gc_p})=x_0^p+\dots +x_{p-1}^p+ p F(x_0,\dots,x_{p-1}),
\]
where $F\in \BZ[x_0,\dots,x_{p-1}]$. We note that Lehmer's argument applies to 
$\per(\eus M_{\gc_p})$ as well. 
\end{rmk}

\begin{prop}    \label{prop:sdvig}
Suppose that $\me$ is a monomial in $\per(\eus M_G)$ such that  $x_k$ occurs in $\me$. If 
$x_k=x_i\dotplus x_j$ for some $i,j$, then there is  $\sigma \in\eus S_n$ such that
$\sigma(i)=j$ and $\me=\boldsymbol{x}(\sigma)$.
\end{prop}
\begin{proof}
By the assumption on $\me$, there is a $\pi\in\eus S_n$ such that 
$\me=\boldsymbol{x}(\pi)$ and $x_k=x_\ap\dotplus x_\beta$  for some $\ap,\beta$ with $\pi(\ap)=\beta$. If $\{\ap,\beta\}\ne \{ i,j\}$, then we have to correct $\pi$.  Take $\gamma\in G$
such that $x_i\dotplus \gamma =x_\ap$. Then $x_\beta\dotplus \gamma=x_j$
 and for $\sigma=\gamma{\ast}\pi$ we have
\[
  \sigma(x_i)= \sigma_\gamma\pi\sigma_\gamma(x_i)= \sigma_\gamma\pi(x_\ap)= \sigma_\gamma(x_\beta)
   =x_\beta\dotplus \gamma=x_j .   
\]
Thus, $\sigma(i)=j$ and also $\boldsymbol{x}(\sigma)=\boldsymbol{x}(\pi)$ in view of Proposition~\ref{prop:nat-action}.
\end{proof}

\noindent
The  Cayley table of $G$ is  the ``addition table'' of all elements of $G$.
Define the {\it extended Cayley table} as an $n{+}1$ by $n{+}1$ matrix
that is the ``addition table'' of $n{+}1$ elements of $G$, with the 
neutral element  taken twice. More precisely, we assume that $x_0=x_n=0$ is the
neutral element of $G$ and consider the matrix $\widetilde{\eus M}_G=(m_{i,j})$, where
$m_{i,j}=x_i\dotplus x_j$, $i,j=0,1,\dots,n$. In this context,  $\eus S_{n+1}$ is regarded
as  permutation group on $\{0,1,\dots,n\}$. Then
$\per(\widetilde{\eus M}_G)=\sum_{\tilde\pi\in \eus S_{n+1}} \boldsymbol{x}(\tilde\pi)$ 
is a sum of monomials of degree $n+1$.

\begin{exa}
$\widetilde{\eus M}_{\gc_3}=\begin{pmatrix}  x_0 & x_1 & x_2 & x_0\\
x_1 & x_2 & x_0 & x_1 \\x_2 & x_0 & x_1 & x_2\\ x_0 & x_1 & x_2 & x_0\\
\end{pmatrix}$, $\per(\widetilde{\eus M}_{\gc_3})=2x_0^4+10x_0^2x_1x_2+4x_0x_1^3+4x_0x_2^3+4x_1^2x_2^2$.
\end{exa}

\begin{thm}   \label{thm:per-ext-keli}
The monomial\/ $\me=\prod_{i=0}^{n-1}x_i^{k_i}$ 
occurs in\/ $\per(\widetilde{\eus M}_G)$ if and only if
\[ \sum_i k_i=n+1 \ \text{ and } \ 
k_0x_0\dotplus \dots \dotplus k_{n-1}x_{n-1}= 0\in G .
\]
\end{thm}
\begin{proof}
"$\Rightarrow$".  Suppose $\me=\boldsymbol{x}(\tilde\pi)$ for some $\tilde\pi\in\eus S_{n+1}$.
Obviously, $\deg\me=n+1$.  Next, 
\[
  k_0x_0\dotplus \dots \dotplus k_{n-1}x_{n-1}=(x_0\dotplus x_{\tilde\pi(0)})\dotplus 
(x_1\dotplus x_{\tilde\pi(1)})\dotplus  \dots \dotplus (x_n\dotplus x_{\tilde\pi(n)})=0,
\] 
since the multiset $\{x_0,x_1,\dots,x_{n-1},x_n=x_0\}$ is closed with respect to taking
inverses. 

"$\Leftarrow$".   Suppose $\me$ satisfies the conditions of the theorem.

\textbullet \ \ $k_0>0$.   Take $\me'=x_0^{k_0-1}x_1^{k_1}\dots x_{n-1}^{k_{n-1}}$. Then $\me'$ satisfies the conditions of Theorem~\ref{thm:hall-perm}. Therefore $\me'$ is a monomial of
$\per(\eus M_G)$ and there is a  $\pi\in \eus S_n$
such that $\me'=\boldsymbol{x}(\pi)$.
Embed $\eus S_n$ into $\eus S_{n+1}$ as the subgroup  preserving  the last element $n$.
Let $\tilde\pi$ denote $\pi$  considered as element of $\eus S_{n+1}$.
Then $\me=\boldsymbol{x}(\tilde\pi)$.

\textbullet \ \ $k_0=0$.   Choose any binomial $x_ix_j$ in $\me$ and replace it with 
$(x_i\dotplus x_j)x_0=x_kx_0$ (i.e., $x_i\dotplus x_j=x_k$).  That is,
$\me=\me''x_ix_j$ is replaced with $\me''x_kx_0=:\me' x_0$. By the previous argument,
we can find $\pi\in \eus S_{n}$  such that  $\me'=\boldsymbol{x}(\pi)$
and $\me' x_0=\boldsymbol{x}(\tilde\pi)$.
Since $x_k=x_i\dotplus x_j$ occurs in $\boldsymbol{x}(\pi)$, we can apply 
Proposition~\ref{prop:sdvig}  and assume that $\pi(i)=j$ and hence $\tilde\pi(i)=j$. 
Finally, we replace $\tilde\pi$ with $\tilde\pi \tau$, where the transposition
$\tau\in \eus S_{n+1}$ permutes $i$ and $n$. One readily verifies that $\boldsymbol{x}(\tilde\pi\tau)=\me''x_ix_j=\me$.
\end{proof}

\begin{cl}
The number of different monomials in\/ $\per(\widetilde{\eus M}_G)$ equals\/ 
$\dim (\cs^{n+1}\eus R)^G$.
\end{cl}
The  proof is almost identical to that of Theorem~\ref{thm:number-in-perm} and left to the reader.

It follows from Frobenius' theory of group determinants (see e.g. \cite[\S\,2]{johnson}) that, 
for abelian groups, $\det(\eus M_{G})$ is the product of linear forms in $x_i$'s. 
In case of generic circulant matrices, this fact
plays an important role in  \cite{lehm} and \cite{thomas}.
For future use, we provide a quick derivation.
Recall that $G=\{x_0,x_1,\dots,x_{n-1}\}$ and $\hat G=\{\chi_0,\chi_1,\dots,\chi_{n-1}\}$.
Consider the $n$ by $n$ complex matrix $\eus K_G$, with $(\eus K_G)_{i,j}=(\chi_j(x_i))$, 
and the vectors $v_j=\sum_{i=0}^{n-1}\chi_j(x_i)x_i\in \eus R$, $j=0,1,\dots,n-1$. 

\begin{prop}   \label{prop:lin-forms-det} 
Under the above notation, we have:
\begin{enumerate}
\item  $v_j$ is an eigenvector of $G$ corresponding to the weight $\chi_j^{-1}$;
\item   $\det(\eus M_G){\cdot} \det(\eus K_G)=\det(\ov{\eus K_G})v_0v_1\dots v_{n-1}$,
where `bar' stands for the complex conjugation;
\item  $\det(\ov{\eus K_G})/ \det(\eus K_G)$ equals the sign of the permutation 
$\pi_0\in\eus S_n$
that takes each $x_i$ to its inverse. Hence $\det(\eus M_G)=\text{\rm sign}(\pi_0)
v_0v_1\dots v_{n-1}$.
\end{enumerate}
\end{prop}
\begin{proof}
(1)  Obvious.

(2) \ It is easily seen that $(\eus M_G{\cdot}\eus K_G)_{ij}=\chi_j(x_i)^{-1}v_j=
\ov{\chi_j(x_i)}v_j=(\ov{\eus K_G})_{i,j} v_j$.

(3) \ Assuming that $x_0$ is the neutral element, 
compare the coefficient of $x_0^n$ in both parts of the equality in (2). 
\end{proof}

Note that $\widetilde{\eus M}_G$ has  equal columns and hence
$\det(\widetilde{\eus M}_G)=0$.

\begin{rmk}
1. The set of vectors $\{v_j\}$ is closed with respect to complex conjugation, and letting
$z_j=\ov{v_j}=\sum_i \ov{\chi_j(x_i)} x_i$ one obtains the eigenvector corresponding to 
$\chi_j$.

2. The orthogonality relations for the characters imply that 
$\eus K_G (\ov{\eus K_G})^t= n \odin_n$; that is, $\frac{1}{\sqrt n}\eus K_G$ is unitary and $|\det(\eus K_G)|^2=n^n$.
\end{rmk}
For the sake of completeness, we mention some other easy properties.

\begin{prop}    \label{prop:izi-prop}
Suppose $\gamma\in G$ and $\pi\in \eus S_n$.
\begin{enumerate}
\item $\gamma{\cdot}\boldsymbol{x}(\pi)=\boldsymbol{x}(\pi\sigma_\gamma^{-1})$, where 
`$\cdot$' stands for the natural $G$-action on $\cs^{n}\eus R$; 
\item $\boldsymbol{x}(\pi)=\boldsymbol{x}(\pi^{-1})$;
\item $\per(\eus M_G)\in (\cs^n\eus R)^G$;
\item If $\hat G$ has a unique element of order 2, say $\psi$, then $\det(\eus M_G)$
is a semi-invariant of weight $\psi$. In all other cases, $\det(\eus M_G)\in (\cs^n\eus R)^G$.
\end{enumerate}
\end{prop}\begin{proof}
1)   $\gamma{\cdot}\boldsymbol{x}(\pi)=\gamma{\cdot}\prod_{i=0}^{n-1}(x_i\dotplus x_{\pi(i)})=
\prod_{i=0}^{n-1}(x_i\dotplus x_{\pi(i)}\dotplus \gamma)=
\prod_{i=0}^{n-1}(x_{\sigma_\gamma(i)}\dotplus x_{\pi(i)})=
\boldsymbol{x}(\pi\sigma_\gamma^{-1})$.

2)  Obvious.

3)  follows from  (1).

4)  Proposition~\ref{prop:lin-forms-det} shows that $\det(\eus M_G)$ is a semi-invariant whose weight
equals the sum of all elements of $\hat G$.
The sum of all elements of an abelian group is known to be the neutral element unless 
the group  has a unique element of order 2, in which case the sum is this unique element.
\end{proof}

Note that $\per(\widetilde{\eus M}_G)$ is an element of $\cs^{n+1}\eus R$, but it 
does not belong to $(\cs^{n+1}\eus R)^G$.

\section{Some open problems}
\label{sect:questions}

\noindent
Associated with previous results on $\per(\eus M_G)$, 
there are some interesting problems. 
Let $d(G)$ denote the number of different monomials in $\det(\eus M_G)$.
In view of possible cancellations, we have $d(G)\le p(G)$. 
Using the factorisation of $\det(\eus M_{\gc_n})$ and  theory of symmetric functions,
Thomas~\cite{thomas} proved
that  $d(\gc_n)=p(\gc_n)$ whenever $n$ is a prime power.
He also computed these values up to $n=12$ (e.g. $d(\gc_6)=68<80 =p(\gc_6)$)
and suggested that the converse could be true. 

\begin{qtn}   \label{vopr:1}
What are necessary/sufficient conditions on a finite abelian group $G$ for the equality 
$d(G)=p(G)$? Specifically, is it still true that the condition `$\#(G)$ is a prime power' is sufficient?
\end{qtn}

The equality $\det(\eus M_G)=\text{\rm sign}(\pi_0)
v_0v_1\dots v_{n-1}$  might be helpful in resolving 
Problem~\ref{vopr:1}. The following problem is more general and vague.

\begin{qtn}
Let $\me$ be a monomial that satisfies conditions of Theorem~\ref{thm:hall-perm}. Is there a group-theoretic 
(or invariant-theoretic) interpretation
of the coefficient of\/ $\me$ in $\per(\eus M_G)$ or $\det(\eus M_G)$?
\end{qtn}

\noindent 
For $G=\gc_2\oplus\gc_2$, we have $p(G)=d(G)=11$.
Because $p(\gc_4)=d(\gc_4)=10$, one may  speculate that 
$p(G)\ge p(\gc_n)$ and $d(G)\ge d(\gc_n)$ if $\#(G)=n$.
By Theorem~\ref{thm:number-in-perm}, $p(G)$ is the coefficient
of $t^{n}$ in the Poincar\'e series  $\cf( (\cs^{\cdot}\eus R)^G;t)$, and one can consider a 
related 

\begin{qtn}
Is it true that \ $[t^m]\cf( (\cs^{\cdot}\eus R)^G;t)\ge [t^m]\cf( (\cs^{\cdot}\eus R)^{\gc_n};t)$
\ for any $m\in \BN$ ?
\end{qtn}

Given $G=\{x_0,\dots,x_{n-1}\}$, a family of matrices 
$\eus M_{G,l}\in {\rm Mat}_l(\BC[x_0,\dots,x_{n-1}])$, $l\ge n$, 
is said to be {\it admissible}, if  $\eus M_{G,n}=\eus M_G$,
$\eus M_{G,l}$ is a principal submatrix of $\eus M_{G,l+1}$, and the number of different monomials in $\per(\eus M_{G,l})$ equals $\dim (\cs^l\eus R)^G$.

\begin{qtn} For what  $G$, does an admissible family exist?
\end{qtn}
So far, we only have matrices $\eus M_{G,l}$ for $l=n, n+1$. 
It is possible to jump up to $l=2n$ by letting
$\eus M_{G,2n}=\begin{pmatrix}  \eus M_G & \eus M_G \\ \eus M_G & \eus M_G 
\end{pmatrix}$. It is the addition table for two consecutive sets of group elements, and it 
can be proved that $\per(\eus M_{G,2n})$ has the required property. Then, similarly to the construction of the extended Cayley table, one defines a larger matrix $\eus M_{G,2n+1}$. This procedure can be iterated, so one obtains a suitable collection of matrices of orders $kn, kn+1$, $k\in\BN$. However, it is not clear whether it is possible to define  matrices $\eus M_{G,l}$ for all other $l$. Maybe the reason is that, for arbitrary abelian $G$, there is no natural ordering of its elements. But, for a cyclic group, one does have a natural ordering, and we provide a conjectural definition of an admissible family of matrices.

For $G=\gc_n$, it will be convenient to begin with the circulant matrix in the Toeplitz form, see Example~\ref{ex:circulant}. That is to say,  our initial matrix is 
$\hat{\eus M}_{\gc_n}=(\hat m_{i,j})$, where  $\hat m_{i,j}=x_{i-j}$, $i,j=0,1,\dots,n-1$,
and the subscripts of $x$'s are interpreted $\pmod n$.
For any $l\ge n$,  we then define the entries of $\hat{\eus M}_{\gc_n, l}$ by the same formula, only the range of $i,j$ is extended. In particular, $\hat{\eus M}_{\gc_n, l}$ is a Toeplitz matrix for any $l$.
\begin{exa}
$\hat{\eus M}_{\gc_3,5}=\begin{pmatrix}  x_0 & x_1 & x_2 & x_0 & x_1\\
x_2 & x_0 & x_1 & x_2 & x_0  \\x_1 & x_2 & x_0 & x_1 & x_2\\ x_0 & x_1 & x_2 & x_0
& x_1 \\ x_2 & x_0 & x_1 & x_2 & x_0
\end{pmatrix}$
\end{exa}
\begin{conj}   \label{conj:quasi-circ}
For $l\ge n$, the monomial  $x_0^{\lb_0}x_1^{\lb_1}\dots x_{n-1}^{\lb_{n-1}}$ occurs in\/ 
$\per(\hat{\eus M}_{\gc_n,l})$ 
if and only if  \\
\hbox to \textwidth{ \ $(\ast)$ 
\hfil 
$\lb_0+\dots +\lb_{n-1}=l   \quad \text{ and } \quad 
 \displaystyle \sum_{j=1}^{n-1} j\lb_j \equiv 0 \mod n$. \hfil } 
In particular, 
the number of different monomials in $\per(\hat{\eus M}_{\gc_n,l})$ equals 
$a_0(\gc_n,l)$.
\end{conj}

It is not hard to verify  the necessity of $(\ast)$  and that the conjecture is true for $n=2$.


\begin{thebibliography}{33}

\bibitem{alm82}
{\sc G.~Almkvist}.
Some formulas in invariant theory, 
{\it J. Algebra} {\bf 77}, no.\,2 (1982),  338--359. 

\bibitem{AF76}
{\sc G.~Almkvist} and {\sc R.~Fossum}. 
Decomposition of exterior and symmetric powers of indecomposable $\BZ/p\BZ$-modules 
in characteristic $p$ and relations to invariants. {\it S\'eminaire d'Alg\`ebre P. Dubreil 
(Paris 1976--77)}, pp. 1--111, Lecture Notes in Math. {\bf 641}, 
Springer-Verlag, Berlin, 1978. 

\bibitem{bru-new}
{\sc R.A.~Brualdi} and {\sc M.~Newman}. 
An enumeration problem for a congruence equation. 
{\it J. Res. Nat. Bur. Standards, Sect. B} {\bf 74B}\,(1970),  37--40.

\bibitem{elji-1} 
{\sc A.G.~Elashvili} and {\sc M.~Jibladze}.
Hermite reciprocity for the regular representations of cyclic groups, 
{\it Indag. Math.} {\bf 9}, no.\,2 (1998), 233--238.

\bibitem{elji-2} 
{\sc A.G.~Elashvili} and {\sc M.~Jibladze}. 
``Hermite reciprocity'' for  semi-invariants in the regular representations of cyclic groups,
{\it Proc. Razmadze Math. Inst.} (Tbilisi),  Vol.~{\bf 119}\,(1999),  21--24.

\bibitem{EJP} 
{\sc A.G.~Elashvili, M.~Jibladze}, and {\sc D.~Pataraia}. 
Combinatorics of necklaces and Hermite reciprocity,
{\it J. Alg. Combinatorics}, {\bf 10} (1999), 173--188.

\bibitem{fred}
{\sc M.~Fredman}. A symmetry relationship for a class of partitions, 
{\it J. Combin. Theory, Ser. A\/} {\bf 18}(1975), 199--202.

\bibitem{hall}
{\sc M.~Hall}.
A combinatorial problem on abelian groups. 
{\it Proc. Amer. Math. Soc.} {\bf 3}(1952), 584--587. 

\bibitem{ha-wr}  
{\sc G.H.~Hardy} and {\sc E.M.~Wright}. ``An introduction to the theory of numbers''. Fifth edition. The Clarendon Press, Oxford University Press, New York, 1979. xvi+426 pp.

\bibitem{johnson}
{\sc K.W.~Johnson}. On the group determinant, 
{\it Math. Proc. Camb. Phil. Soc.} {\bf 109} (1991), 299--311.

\bibitem{lehm}
{\sc D.~Lehmer}.  Some properties of circulants, 
{\it J. Number Theory}  {\bf 5}(1973), 43--54.

\bibitem{st79}
{\sc R.P.~Stanley}.
Invariants of finite groups and their applications to combinatorics, 
{\it  Bull. Amer. Math. Soc.} (N.S.) {\bf 1}, no.\,3 (1979), 475--511. 

\bibitem{thomas}
{\sc H.~Thomas}.
The number of terms in the permanent and the determinant of a generic circulant matrix,
{\it J. Algebraic Combin.}, {\bf 20} (2004), 55--60.

\end{thebibliography}
\end{document}